\documentclass[reqno,12pt]{amsart}
\usepackage[english]{babel}
\usepackage{amsmath,amsfonts,amsthm,amssymb,amsbsy,upref,color,graphicx,hyperref,enumerate,comment}
\usepackage[a4paper,margin=2.84truecm]{geometry}
\usepackage{bbm}
\usepackage{soul}

\DeclareMathOperator{\dive}{div}

\def\ds{\displaystyle}
\def\eps{{\varepsilon}}

\def\N{\mathbb{N}}
\def\O{\Omega}
\def\R{\mathbb{R}}
\def\A{\mathcal{A}}

\def\HH{\mathcal{H}}
\def\F{\mathcal{F}}
\def\PP{\mathcal{P}}
\def\pa{\partial}
\newcommand{\be}{\begin{equation}}
\newcommand{\ee}{\end{equation}}
\newcommand{\bib}[4]{\bibitem{#1}{\sc#2: }{\it#3. }{#4.}}

\numberwithin{equation}{section}
\theoremstyle{plain}

\newtheorem{theo}{Theorem}[section]
\newtheorem{lemm}[theo]{Lemma}
\newtheorem{coro}[theo]{Corollary}
\newtheorem{prop}[theo]{Proposition}

\theoremstyle{definition}
\newtheorem{rem}[theo]{Remark}

\title[A shape optimization problem on planar sets with prescribed topology]{A shape optimization problem on planar sets with prescribed topology}

\author[L. Briani]{Luca Briani}

\author[G. Buttazzo]{Giuseppe Buttazzo}

\author[F. Prinari]{Francesca Prinari}

\date{}

\begin{document}

\maketitle

{\it Dedicated to Franco Giannessi for his 85th birthday}

\begin{abstract}
We consider shape optimization problems involving functionals depending on perimeter, torsional rigidity and Lebesgue measure. The scaling free cost functionals are of the form $P(\O)T^q(\O)|\O|^{-2q-1/2}$ and the class of admissible domains consists of two-dimensional open sets $\O$ satisfying the topological constraints of having a prescribed number $k$ of bounded connected components of the complementary set. A relaxed procedure is needed to have a well-posed problem and we show that when $q<1/2$ an optimal relaxed domain exists. When $q>1/2$ the problem is ill-posed and for $q=1/2$ the explicit value of the infimum is provided in the  cases $k=0$ and $k=1$.
\end{abstract}



\section{Introduction\label{sintro}}

In the present paper we aim to study some particular shape optimization problems in classes of planar domains having a prescribed topology. The quantities we are going to consider for a general bounded open set $\O $ are the distributional perimeter $P(\O)$ and the torsional rigidity $T(\O)$. More precisely, we deal with a scaling free functional $F_q$ which is expressed as the product of the perimeter, and of a suitable powers of the torsional rigidity and of the Lebesgue measure of $\O$, depending on a positive parameter $q$.
The restriction to the planar case is essential and is not made here for the sake of simplicity; indeed, in higher dimension stronger topological constraints have to be imposed to make the problems well posed.

In a previous paper \cite{BBP20} we treated the problem above in every space dimension and, after discussing it for general open sets, we focused to the class of convex open sets.

In the following we consider the optimization problems for $F_q$ in the classes $\A_k$ of planar  domains having at most $k$ ``holes".

While the maximization problems are always ill posed, even in the class of smooth open sets in $\A_k$, it turns out that the minimizing problems are interesting if $q\le 1/2$ and some regularity constraints are imposed to the sets $\O\in\A_k$.

In this case, we provide a explicit lower bound for $F_q$ in the class of Lipschitz sets in $\A_k$, which turns out to be sharp when $k=0,1$ and $q=1/2$ and coincides with the infimum of $F_q$ in the class of convex sets, as pointed out by Polya in \cite{polya60}.

When $q<1/2$ we study the existence of minimizers for $F_q$ and our approach is the one of direct methods of the calculus of variations which consists in the following steps:
\begin{itemize}
\item[-]defining the functional $F_q$ only for Lipschitz domains of the class $\A_k$;
\item[-]relaxing the functional $F_q$ on the whole class $\A_k$, with respect to a suitable topology;
\item[-]showing that the relaxed functional admits an optimal domain in $\A_k$;
\item[-]proving that such a domain is Lipschitz.
\end{itemize}
The relaxation procedure above is necessary to avoid trivial counterexamples due to the fact that the perimeter is Lebesgue measure sensitive, while the torsional rigidity is capacity sensitive.

As in most of the free boundary problems, the last regularity step presents strong difficulties and, even if the regularity of optimal domains could be expected, we are unable to give a full proof of this fact. It would be very interesting to establish if an optimal domain fulfills some kind of regularity, or at least if its perimeter  coincides with the Hausdorff measure of the boundary, which amounts to exclude the presence of internal fractures.

This paper is organized as follows.  
In Section \ref{spre}, after recalling the definitions of perimeter and torsional rigidity,  we summarize the main results of this paper. In Section \ref{sapp} we describe the key tools necessary to apply the so-called method of \textit{interior parallels}, introduced by Makai in \cite{Ma},\cite{Ma59} and by Polya in \cite{polya60}, to our setting. Section \ref{shau} contains a review of some basic facts concerning the complementary Hausdorff convergence, with respect to which we perform the relaxation procedure. Although Sections \ref{sapp} and \ref{shau} may be seen as preliminary, we believe they contain some interesting results that, as far as we know, are new in  literature. Finally, in Section \ref{sexis} we discuss the optimization problem: we extend a well known inequality due to Polya (Theorem \ref{theo.Polya} and Remark \ref{rem.polya}), and we prove the main results (Corollary \ref{coro.polya} and Theorem \ref{theo.exis}).

\section{Preliminaries}\label{spre}

The shape functionals we consider in this paper are of the form
\begin{equation}\label{Fq}F_q(\O)=\frac{P(\O)T^q(\O)}{|\O|^{2q+1/2}}
\end{equation}
where $q>0$, $\O\subset\R^2$  is a general bounded open set and, $|\O|$ denotes its Lebesgue measure. 

For the reader's convenience,  in the following  we report  the definitions and the basic properties of the perimeter and of the torsional rigidity. According to the De Giorgi formula, the perimeter is given by
$$P(\O)=\sup\left\{\int_\O\dive\phi\,dx\ :\ \phi\in C^1_c(\R^2;\R^2),\ \|\phi\|_{L^\infty(\R^2)}\le1\right\},$$
and satisfies:
\begin{itemize}
\item[-]the {\it scaling property}
$$P(t\O)=tP(\O)\qquad\text{for every }t>0;$$
\item[-] the lower semicontinuity with respect to the $L^1$-convergence, that is the convergence of characteristic functions.
\item[-]the {\it isoperimetric inequality}
\be\label{isoper}
\frac{P(\O)}{|\O|^{1/2}}\ge\frac{P(B)}{|B|^{1/2}}
\ee
where $B$ is any disc in $\R^2$. In addition the inequality above becomes an equality if and only if $\O$ is a disc (up to sets of Lebesgue measure zero).
\end{itemize}

The torsional rigidity $T(\O)$ is defined as
$$T(\O)=\int_\O u\,dx$$
where $u$ is the unique solution of the PDE
\be\label{pdetorsion}\begin{cases}
-\Delta u=1&\text{in }\O,\\
u\in H^1_0(\O).
\end{cases}
\ee
By means of an integration by parts we can equivalently express the torsional rigidity as
\be \label{vartor} 
T(\O)=\max\Big\{\Big[\int_\O u\,dx\Big]^2\Big[\int_\O|\nabla u|^2\,dx\Big]^{-1}\ :\ u\in H^1_0(\O)\setminus\{0\}\Big\}.
\ee
The main properties we use for the torsional rigidity are:
\begin{itemize}
\item[-]the monotonicity with respect to the set inclusion
$$\O_1\subset\O_2\Longrightarrow T(\O_1)\le T(\O_2);$$
\item[-]the additivity on disjoint families of open sets
$$T\Big(\bigcup_n\O_n\Big)=\sum_n T(\O_n)\qquad\text{whenever $\O_n$ are pairwise disjoint;}$$
\item[-]the scaling property
$$T(t\O)=t^4T(\O),\qquad\text{for every }t>0;$$
\item[-]the relation between torsional rigidity and Lebesgue measure (known as {\it Saint-Venant inequality})
\be\label{stven}
\frac{T(\O)}{|\O|^2}\le\frac{T(B)}{|B|^2}.
\ee
In addition, the inequality above becomes an equality if and only if $\O$ is a disc (up to sets of capacity zero).
\end{itemize}

If we denote by $B_1$ the unitary disc of $\R^2$, then the solution of \eqref{pdetorsion}, with $\O=B_1$, is
$$u(x)=\frac{1-|x|^2}{4}$$
which provides
$$T(B_1)=\frac{\pi}{8}.$$

Thanks to the scaling properties of the perimeter and of the torsional rigidity, the functional $F_q$ defined  by \eqref{Fq} is {\it scaling free} and optimizing it in a suitable class $\A$ is equivalent to optimizing the product $P(\O)T^q(\O)$ over $\A$ with the additional measure constraint $|\O|=m$, for a fixed $m>0$.

In a previous paper \cite{BBP20} we considered the minimum and the maximum problem for $F_q$ (in every space dimension) in the classes
\[\begin{split}
&\A_{all}:=\big\{\O\subset\R^d\ :\ \O\ne\emptyset\big\}\\
&\A_{convex}:=\big\{\O\subset\R^d\ :\ \O\ne\emptyset,\ \O\text{ convex}\big\}.
\end{split}\]
We summarize here below the results available in the case of dimension 2:
\begin{itemize}
\item[-]for every $q>0$
$$\inf\big\{F_q(\O)\ :\ \O\in\A_{all},\ \O\text{ smooth}\big\}=0;$$
\item[-]for every $q>0$
$$\sup\big\{F_q(\O)\ :\ \O\in\A_{all},\ \O\text{ smooth}\big\}=+\infty;$$
\item[-]for every $q>1/2$
$$\begin{cases}
\inf\big\{F_q(\O)\ :\ \O\in\A_{convex}\big\}=0\\
\max\big\{F_q(\O)\ :\ \O\in\A_{convex}\big\}\quad\text{is attained};
\end{cases}$$
\item[-]for every $q<1/2$
$$\begin{cases}
\sup\big\{F_q(\O)\ :\ \O\in\A_{convex}\big\}=+\infty\\
\min\big\{F_q(\O)\ :\ \O\in\A_{convex}\big\}\quad\text{is attained};\\
\end{cases}$$
\item[-]for $q=1/2$
$$\begin{cases}
\inf\big\{F_{1/2}(\O)\ :\ \O\in\A_{convex}\big\}=(1/3)^{1/2}\\
\sup\big\{F_{1/2}(\O)\ :\ \O\in\A_{convex}\big\}=(2/3)^{1/2},
\end{cases}$$
asymptotically attained, respectively, when $\O$ is a long thin rectangle and when $\O$ is a long thin triangle.
\end{itemize}

Here we discuss the optimization problems for $F_q$ on the classes of planar domains
$$\A_k:=\big\{\O\subset\R^2\ :\ \O\ne\emptyset,\ \O\text{ bounded, }\#\O^c\le k\big\},$$
where, for every set $E$, we denote by $\#E$ the number of bounded connected components of $E$ and $\O^c=\R^2\setminus\O$. In particular $\A_0$ denotes the class of simply connected domains (not necessarily connected).

From what seen above the only interesting cases to consider are:
$$\begin{cases}
\text{the maximum problem for $F_q$ on $\A_k$ when $q\ge1/2$ ;}\\
\text{the minimum problem for $F_q$ on $\A_k$ when $q\le1/2$.}
\end{cases}$$
We notice that the maximum problem is not well posed, since for every $q>0$ and every $k\ge0$
$$\sup\big\{F_q(\O)\ :\ \O\text{ smooth},\ \O\in\A_k\big\}=+\infty.$$
Indeed, it is enough to take as $\O_n$ a smooth perturbation of the unit disc $B_1$ such that
$$B_{1/2}\subset\O_n\subset B_2\qquad\text{and}\qquad P(\O_n)\to+\infty.$$
All the domains $\O_n$ are simply connected, so belong to $\A_k$ for every $k\ge0$, and
$$|\O_n|\le|B_2|,\qquad T(\O_n)\ge T(B_{1/2}),$$
where we used the monotonicity of the torsional rigidity. Therefore
$$F_q(\O_n)\ge\frac{P(\O_n)T^q(B_{1/2})}{|B_2|^{2q+1/2}}\to+\infty.$$
Moreover  $$\inf\big\{F_q(\O)\ :\ \O\in\A_k\big\}=0,$$ as we can easily see by taking as $\O_n$ the unit disk of $\R^2$ where we remove the $n$ segments (in polar coordinates $r,\theta$)
$$S_i=\big\{\theta=2\pi i/n,\ r\in[1/n,1]\big\}\qquad i=1,\dots,n.$$
We have that all the $\O_n$ are simply connected, and
$$|\O_n|=\pi,\qquad P(\O_n)=2\pi,\qquad T(\O_n)\to0,$$
providing then $F_q(\O_n)\to0$. 

Therefore, the problems we study in the sequel are 
$$\inf\big\{F_q(\O)\ :\ \O\in \A_k,\text{ $\O$ Lipschitz}\},$$
when $q\le1/2$ and $k\in\N$. Denoting by $m_{q,k}$ the infimum above we summarize here below our main results.

\begin{itemize}
\item[-] For every $q\le1/2$ the values $m_{q,k}$ are decreasing with respect to $k$ and
$$\lim_{k\to\infty}m_{q,k}=0.$$
\item[-]When $k=0,1$ it holds
$$m_{1/2,0}=m_{1/2,1}=3^{-1/2}=\inf\big\{F_{1/2}(\O)\ :\ \O\text{ convex}\big\};$$
in particular, for $q=1/2$ there is no gap for $\inf F_{1/2}$ between the classes $\A_{convex}$, $\A_0$, $\A_1$, and the infimum is asymptotically reached by a sequence of long and thin rectangles.
\item[-]For every $q\le1/2$ and $k\in\N$, we have 
$$m_{q,k}\ge\begin{cases}(8\pi)^{1/2-q}3^{-1/2}&\text{if }k=0,1, \\
(8\pi)^{1/2-q}(3^{1/2}k)^{-1}&\text{if }k>1.
\end{cases}$$
\item[-]For $q<1/2$, we define a relaxed functional $\F_{q,k}$, which coincides with $F_q$ in the class of the sets $\O\in\A_k$ satisfying $P(\O)=\HH^1(\pa\O)$, being $\HH^1$ the $1$-dimensional Hausdorff measure.
We also prove that $\F_{q,k}$ admits an optimal domain $\O^{\star}\in\A_k$ with $\HH^{1}(\pa\O^\star)<\infty$.
\end{itemize}

\section{ Approximation by interior parallel sets} \label{sapp}

For a given bounded nonempty open set $\O$ we denote by $\rho(\O)$ its \textit{inradius}, defined as
$$\rho(\O):=\sup\big\{d(x,\pa\O)\ :\ x\in\O\big\},$$
where, as usual,
$d(x,E):=\inf\big\{d(x,y)\ :\ y\in E\big\}$.
For every $t\ge 0$, we denote by $\O(t)$ the \textit{interior parallel set} at distance $t$ from $\pa\O$, i.e.
$$\O(t):=\big\{x\in\O\ :\ d(x,\pa\O)>t\big\},$$
and by $A(t):=|\O(t)|$. Moreover we denote by $L(t)$ the length of the \textit{interior parallel}, that is the set of the points in $\O$ whose distance from $\pa\O$ is equal to $t$.
More precisely we set 
$$L(t):=\HH^1 (\{x\in\O\ :\ d(x,\pa\O)=t \}).$$
Notice that  $\partial \O(t)\subseteq \{x\in\O\ :\ d(x,\pa\O)=t \}$. 
Using coarea formula (see \cite{EvGa} Theorem 3.13) we can write the following identity:
\be\label{eq.Evans}
A(t)=\int_t^{\rho(\O)}L(s)\,ds\qquad\forall t\in(0,\rho(\O)).
\ee As a consequence it is easy to verify that for a.e. $t\in(0,\rho(\O))$ there exists the derivative $A'(t)$ and it coincides with $-L(t)$. The interior parallel sets $\O(t)$ belong to $\A_k$ as soon as $\O\in\A_k$, as next elementary argument shows.

\begin{lemm}\label{lemm.innerA_k} Let $\Omega\in\A_k$. Then $\O(t)\in \A_k$ for every $t\in [0,\rho(\O))$.
\end{lemm}
\begin{proof} Let $\alpha:=\#\O^c$ ($\le k$), and $C^1,C^2,\cdots C^\alpha$ be the (closed) bounded connected components of $\O^c$ and $C^0$ the unbounded one. Define
$$C^i(t):=\big\{x\in\R^2\ : \ d(x,C^i)\le t\big\}.$$
Since $C^i$ is connected, then $C^i(t)$ is connected and the set $\bigcup_{i=0}^\alpha C^i(t)$ has at most $\alpha+1$ connected components. Since we have
$\O^c(t)=\bigcup_{i=0}^\alpha C^i(t)$,
the lemma is proved.
\end{proof}

In the planar case, even without any regularity assumptions on $\pa\O$, the sets $\O(t)$ are a slightly smoothed version of $\O$. In particular the following result (see \cite{Fu85}), that we limit to report in the two dimensional case, proves that $\O(t)$ has a Lipschitz boundary for a.e. $t\in(0,\rho(\O))$.

\begin{theo}[Fu]\label{theo.Fu}
Let $K\subseteq \R^2$ be a compact set. There exists a compact set $C=C(K)\subseteq [0, 3^{-1/ 2} diam(K)]$ such that $|C|=0$ and if $t\notin C$ then the boundary of $\{x\in\R^2\ :\ d(x,K)>t\}$ is a Lipschitz manifold. 
\end{theo}
We recall now some general facts of geometric measure theory. Let $E\subset \R^2$, we denote by $E^{(t)}$ the set of the points where the density of $E$ is $t\in [0,1]$, that is
$$E^{(t)}:=\{ x\in\R^2: \lim_{r\to 0^+} (\pi r^2)^{-1}|E\cap B_r(x)|=t\}.$$ 
It is well known (see \cite{AFP} Theorem 3.61) that if $E$ is a set of finite perimeter, then $P(E)=\HH^1(E^{1/2})$ and $E$ has density either $0$ or $1/2$ or $1$ at $\HH^1$-a.e $x\in\R^2$. In particular it holds
\be\label{eq.decomp}\HH^1(\partial E)= \HH^1(\partial E\cap E^{(0)})+ \HH^1(\partial E\cap E^{(1)})+ P(E),
\ee
which implies 
\be \label{eq.decomp2}
P(E)+2\HH^1(\pa E\cap E^{(1)})\le 2 \HH^1(\pa E)-P(E).
\ee
The Minkowski content and the outer Minkowski content of $E$ are, respectively, defined as
$$\mathcal{M}(E):=\lim_{t\to 0}\frac {|\{x\in\R^2\ :\ d(x,E)\le t\}|}{2t},$$
and 
$$\mathcal{SM}(E):=\lim_{t\to 0}\frac {|\{x\in\R^2\ :\ d(x,E)\le t\}\setminus E|}{t},$$
whenever the limits above exist. 

We say that a compact set $E\subset\R^2$ is $1$-rectifiable if there exists a compact set $K\subset\R$ and a Lipschitz map $f:\R\to \R^2$ such that $f(K)=E$. Any compact connected set of $\R^2$, namely a \textit{continuum}, with finite $\HH^1$-measure is $1$-rectifiable (see, for instance, Theorem 4.4 in \cite{AO}). Finally, if $E$ is $1$-rectifiable then
\be\label{eq.Amb}
\mathcal{M}(E)=\HH^1(E)
\ee (see Theorem $2.106$ in \cite{AFP})
and by Proposition 4.1 of \cite{V}, if $E$ is a Borel set and $\pa E$ is $1$-rectifiable it holds
\be\label{eq.Villa}
\mathcal{SM}(E)=P(E)+2\HH^1(\pa E\cap E^{(0)}).
\ee
Next two results are easy consequence of \eqref{eq.Amb} and \eqref{eq.Villa}.

\begin{theo}\label{theo.min}
Let $\O$ be a bounded open set with $\HH^1(\pa \O)<\infty$ and $\#\pa\O<+\infty$. Then $
\mathcal{M}(\pa\O)=\HH^{1}(\pa \O)$ and $ \mathcal{SM}(\O)=P(\O)+2 \HH^1(\partial \O\cap \O^{(0)}).$
\end{theo}

\begin{proof}
Since $\HH^1(\pa\O)<\infty$, each connected component of $\pa\O$ is $1$-rectifiable. Being the connected components pairwise disjoint and compact, we easily prove that their finite union is $1$-rectifiable. Then, applying \eqref{eq.Amb} and \eqref{eq.Villa}, we get the thesis.
\end{proof}

\begin{coro}\label{coro.Mink}
Let $\O$ be an open set such that $\HH^1(\pa\O)<\infty$ and $\#\pa\O<+\infty$. Then there exists
$$\lim_{r\to 0^+}\frac 1 r\int_{0}^r L(t)dt= P(\O)+2 \HH^1(\partial \O\cap \O^{(1)}).$$
\end{coro}

\begin{proof}
 We denote by $L^c(t)$ the following quantity 
$$L^c(t):=\HH^1(\{x\in\O^c\ :\ d(x,\pa\O)=t\}).$$
By applying coarea formula and Theorem \ref{theo.min}, it holds
\be\label{eq.g1}
\lim_{r\to 0^+} \frac{1}{r}\int_{0}^{r}L^c(t)dt=\mathcal{SM}(\O)=P(\O)+2 \HH^1(\partial \O\cap \O^{(0)}).
\ee
and 
\be\label{eq.min}
\lim_{r\to 0^+}\frac{1}{r}\int_{0}^{r}\left[L(t)+L^c(t)\right]dt=2\mathcal{M}(\pa\O)=2\HH^1(\pa \O).
\ee
Combining \eqref{eq.decomp}, \eqref{eq.g1} and \eqref{eq.min} we get 
$$
\lim_{r\to 0^+}\frac{1}{r}\int_{0}^{r}L(t) dt=
\lim_{r\to 0^+}\left(\frac{1}{r}\int_{0}^{r}L(t) dt+\frac{1}{r}\int_{0}^{r}L^c(t) dt-\frac{1}{r}\int_{0}^{r}L^c(t) dt\right)$$
$$= 2\HH^1(\pa \O)-P(\O)-2 \HH^1(\partial \O\cap \O^{(0)})= P(\O)+2 \HH^1(\partial \O\cap \O^{(1)}) $$
and the thesis is achieved.
\end{proof}
Most of the results we present rely on a geometrical theorem proved by Sz. Nagy in \cite{Nagy59}, concerning the behavior of the function $t\to A(t)=|\O(t)|$ for a given set $\O\in\A_k$.
\begin{theo}[Sz. Nagy]\label{theo.Na}
Let $\O\in\A_k$ and let $\alpha:=\#\O^c$. Then the function 
$$t\mapsto-A(t)-(\alpha-1)\pi t^{2}$$
is concave in $[0,\rho(\O))$. 
\end{theo}

As a consequence of Corollary \ref{coro.Mink} and Theorem \ref{theo.Na} we have the following result.

\begin{theo}\label{theo.Nareg}
Let $\O\in\A_k$ with $\HH^1(\pa \O)<\infty$ and $\#\O<+\infty$. Then, for a.e. $t\in(0,\rho(\O))$, it holds: 
\begin{align}\label{eq.boundL}
&L(t)\le P(\O)+2 \HH^1(\partial \O\cap \O^{(1)}) +2\pi(k-1)t;\\
\label{eq.boundA}
&A(t)\le (P(\O)+2 \HH^1(\partial \O\cap \O^{(1)}))(\rho(\O)-t)+\pi(k-1)(\rho(\O)-t)^2.
\end{align}
In particular $A\in W^{1,\infty}(0,\rho(\O))$.
\end{theo}

\begin{proof}
We denote by $g(t)$ the right derivative of the function $t\mapsto -A(t)-(\alpha-1)\pi t^2$ where $\alpha:=\#\O^c$ $(\le k)$. By Theorem \ref{theo.Na}, $g$ is a decreasing function in $(0,\rho(\O))$ and an easy computation through \eqref{eq.Evans} shows that 
\be\label{Lg}
g(t)=L(t)-2\pi(\alpha-1)t\qquad\hbox {for a.e. }t\in(0,\rho(\O)).
\ee
Thus,
$$\lim_{r\to 0^+}\frac 1 r\int_{0}^r L(t)dt=\lim_{r\to 0^+}\frac 1 r\int_{0}^rg(t)dt=\sup_{(0,\rho(\O))}g(t).$$
Since $\O\in\A_k$ and $\#\O<\infty$ we have also $\#\pa\O<\infty$. Hence we can apply Corollary \ref{coro.Mink} to get
\be\label{eq.g}
P(\O)+2 \HH^1(\partial \O\cap \O^{(1)})=\sup_{(0,\rho(\O))}g(t).
\ee
By using \eqref{Lg} and \eqref{eq.g}, inequality \eqref{eq.boundL} easily follows. Finally, by applying \eqref{eq.Evans}, we get both $A\in W^{1,\infty}(0,\rho(\O))$ and formula \eqref{eq.boundA}.
\end{proof}

The following lemma can be easily proved by lower semicontinuity property of the perimeter.
\begin{lemm}\label{lem.top1}
Let $\Omega\subset\R^2$ be an open set. Let $(\Omega^i)$ be its connected components and $\O_n:=\bigcup_{i=1}^n\O^i$. Then we have:
\begin{enumerate}
\item[(i)]$\pa\O_n=\bigcup_{i=1}^n\pa\O^i\subseteq\pa\O$ and $\HH^1(\pa\O_n)\le \HH^1(\pa\O)$;
\item [(ii)]$\ds P(\O)\le\liminf_{n\to\infty}P(\O_n)\le\limsup_{n\to\infty}P(\O_n)\le\limsup_{n\to\infty}\HH^1(\pa\O_n)\le \HH^1(\pa\O)$.
\end{enumerate}
\end{lemm}

We are now in a position to prove the main results of this section. In Theorem 1.1 of \cite{Sc15} it is shown that, given any set $\O$ of finite perimeter satisfying $\HH^1(\pa\O)=P(\O)$, it is possible to approximate $P(\O)$ with the perimeters of smooth open sets compactly contained in $\O$. Here we show that, if we assume the further hypothesis $\O\in\A_k$, then we can construct an approximation sequence made up of Lipschitz sets in $\A_k$.

\begin{theo}\label{theo.approxim}
Let $\O\in\A_k$ be a set of finite perimeter. Then there exists an increasing sequence $(A_n)\subset \A_k$ such that: 
\begin{enumerate}
\item [(i)] $\overline A_n\subset \O$;
\item[(ii)] $\bigcup_{n} A_n=\O$;
\item[(iii)] $A_n$ is a Lipschitz set;
\item[(iv)] $\ds P(\O)\le\liminf_{n\to\infty}P(A_n)\le\limsup_{n\to\infty}P(A_n)\le2\HH^1(\pa\O)-P(\O)$.
\end{enumerate}
In addition, if $\# \O<\infty$, then
$$\lim_{n\to\infty}P(A_n)=P(\O)+2 \HH^1(\partial\O\cap\O^{(1)}).$$
\end{theo}

\begin{proof}
Let $\O_n$ be defined as in Lemma \ref{lem.top1}. Clearly $\O_n\in\A_k$. 
Since $\O_n(t)$ converges to $\O_n$ in $L^1$ when $t\to0^+$, it follows that, for every $n$,
$$\liminf_{t\to 0^+} P(\O_n(t))\ge P(\O_n).$$
Then there exists $0<\delta_n<1/n\wedge \rho(\O_n)$ such that
\be\label{primacond}P(\O_n(t))\ge P(\O_n)-\frac1n\qquad\forall t<\delta_n.\ee
Since $\#\O_n\le n$, by applying Theorem \ref{theo.Fu}, Lemma \ref{lemm.innerA_k} and Theorem \ref{theo.Nareg} to the set $\O_n$, we can choose a decreasing sequence $(t_n)$ with $0<t_n< \delta_n$ such that the set $A_n:=\O_n(t_n)$ is in $\A_k$, has Lipschitz boundary, and
\be\label{secondacond} 
\HH^1(\{x\in\O_n\ :\ d(x,\pa\O_n)=t_n \})\le P(\O_n)+2\HH^1(\pa\O_n\cap\O_n^{(1)})+2\pi(k-1)t_n.
\ee
It is easy to prove that the sequence $(A_n)$ is increasing and satisfies (i) and (ii).
By putting together \eqref{primacond} and \eqref{secondacond}, we get
\begin{equation*}
P(\O_n)-\frac1n\le P(A_n)\le P(\O_n)+2\HH^1(\pa\O_n\cap\O_n^{(1)})+2\pi(k-1)t_n.
\end{equation*}
By Lemma \ref{lem.top1}, taking also into account \eqref{eq.decomp2}, the previous inequality implies
$$P(\O)\le\liminf_n P(A_n)\le\limsup_n P(A_n)\le2\HH^1(\pa\O)-P(\O)$$
which proves $(iv)$.

To conclude consider the case $\#\O<+\infty$. We can choose $n$ big enough such that $\O_n=\O$, $A_n=\O(t_n)$ and $\alpha:=\# \O^c=\# A_n^c$. For simplicity we denote $\rho_n:=\rho(A_n)$ and $\rho:=\rho(\O)$. By applying equality \eqref{eq.g} to the Lipschitz set $A_n$, we get
\be\label{terza}
P(A_n)=\sup_{ (0,\rho_n)}g_n(t)
\ee
where $g_n$ is the right derivative of the function $t\mapsto -|A_n(t)|-(\alpha-1)\pi t^2$.
Now, exploiting the equality $A_n(t)=\O(t+t_n)$, 
we obtain
$$g_n(t)= g(t+t_n)+2\pi(\alpha-1)t_n$$
for all $0<t<(\rho-t_n)\wedge \rho_n$. Thus, as $t\to 0^+$ and applying \eqref{terza}, we can conclude that, for every $n$, it holds
$$\lim_{t\to 0^+}g(t+t_n)+2\pi(\alpha-1)t_n=\sup_{(0,\rho_n)}g_n(t)=P(A_n).$$
Passing to the limit as $n\to\infty$ in the equality above and taking into account \eqref{eq.g} we achieve the thesis.
\end{proof}

\section{Continuity of volume for co-Hausdorff convergence}\label{shau}

The Hausdorff distance between closed sets $C_1, C_2$ of $\R^2$ is defined by
$$d_H(C_1,C_2):=\sup_{x\in C_1}d(x,C_2)\vee\sup_{x\in C_2}d(x,C_1).$$
Through $d_H$ we can define the so called co-Hausdorff distance $d_{H^c}$ between a pair of bounded open subsets $\O_1,\O_2$ of $\R^2$ 
$$d_{H^c}(\O_1,\O_2):=d_H(\O_1^c,\O_2^c).$$
We say that a sequence of compact sets $(K_n)$ converges in the sense of Hausdorff to some compact set $K$, if $ (d_H(K_n,K))$ converges to zero. In this case we write $K_n\overset{H}{\to}K$. Similarly we say that a sequence of open sets $(\O_n)$ converges in the sense of co-Hausdorff to some open set $\O$, if $(d_{H^c}(\O_n,\O))$ converges to zero, and we write $\O_n\overset{H^c}{\to}\O$. In the rest of the paper we use some elementary properties of Hausdorff distance and co-Hausdorff distance for which we refer to \cite{bubu05} and \cite{He}, (see, for instance, Proposition 4.6.1 of \cite{bubu05}). In particular we recall that if $(\O_n)$ is a sequence of equi-bounded sets in $\A_k$ and $\O_n\overset{H^c}{\to}\O$, then $\O$ still belongs to $\A_k$ (see Remark 2.2.20 of \cite{He}). 

The introduction of co-Hausdorff convergence is motivated by Sver\'ak's Theorem (see \cite{sv93}) which ensures the continuity of the torsional rigidity in the class $\A_k$. Actually the result is stronger and gives the continuity with respect to the $\gamma$-convergence (we refer to \cite{bubu05} for its precise definition and the related details).
 
\begin{theo}[Sver\'ak]\label{theo.Sve}
Let $(\O_n)\subset\A_k$ be a sequence of equi-bounded open sets. If $\O_n\overset{H^c}{\to}\O$, then $\O_n\to\O$ in the $\gamma$-convergence. In particular $T(\O_n)\to T(\O)$.
\end{theo}

Combining Sver\'ak theorem and Theorem \ref{theo.approxim}, we prove that we can equivalently minimize the functional $F_q$ either in the class of Lipschitz set in $\A_k$ or in the larger class of those sets $\O\in\A_k$ satisfying $P(\O)=\HH^1(\pa\O)$.

\begin{prop}
The following identity holds:
$$m_{q,k}=\inf\{F_q(\O)\ :\ \O\in\A_k,\ P(\O)=\HH^1(\pa\O)\}$$
\end{prop}

\begin{proof}
By Theorem \ref{theo.approxim}, for every $\O\in\A_k$ such that $P(\O)=\HH^1(\pa\O)<\infty$, there exists a sequence $(A_n)\subset\A_k$ of Lipschitz sets satisfying $\lim_n P(A_n)=P(\O)$. By construction $(A_n)$ is an equi-bounded sequence which converges both in the co-Hausdorff and in the $L^1$ sense. By Theorem \ref{theo.Sve} we have
$$\lim_{n\to\infty} F_q(\O_n)=F_q(\O),$$
so that 
$$m_{q,k}\le\inf\{F_q(\O)\ :\ \O\in\A_k,\ P(\O)=\HH^1(\pa\O)\}.$$
The thesis is then achieved since the opposite inequality is trivial.
\end{proof}

In general the volume is only lower semicontinuous with respect to the $H^c$-convergence as simple counterexamples may show. In this section we prove that $L^1$-convergence is guaranteed in the class $\A_k$ under some further hypotheses, see Theorem \ref{theo.convmeas}. The proof of this result requires several lemma and relies on the classical Go\l ab's semicontinuity theorem, which deals with the lower semicontinuity of the Hausdorff measure $\HH^1$ (see, for instance, \cite{AO}, \cite{amti04}).

\begin{theo}[Go\l ab]\label{theo.Golab}
Let $X$ be a complete metric space and $k\in\N$ let 
$$\mathcal{C}_k:=\{K\ :\ K\subset X, \ K\text{ is closed},\ \#K\le k\}.$$
Then the function $K\mapsto\HH^1(K)$ is lower semicontinuous on $\mathcal{C}_k$ endowed with the Hausdorff distance.
\end{theo}

\begin{lemm}\label{lem.inradcon}
Let $(\O_n)$ be a sequence of equi-bounded open sets. If $\O_n\overset{H^c}{\to}\O$ we have also $\rho(\O_n)\to\rho(\O)$.
\end{lemm}

\begin{proof}
For simplicity we denote $\rho:=\rho(\O)$, and $\rho_n:=\rho(\O_n)$. First we show that \be\label{basso}\rho\le\liminf_n\rho_n.\ee Indeed, without loss of generality let us assume $\rho>0$. Then for any $0<\eps<\rho$, there exists a ball $B_\eps$ whose radius is $\rho-\eps$ and whose closure is contained in $\O$. By elementary properties of co-Hausdorff convergence, there exists $\nu$ such that $B_\eps\subset\O_n$, for $n>\nu$, which implies $\rho_n\ge\rho-\eps$. Since $\eps>0$ is arbitrary, we get \eqref{basso}.

In order to prove the upper semicontinuity, assume by contradiction that there exist $\eps>0$ and a subsequence $(n_k)$ such that $\rho_{n_k}>\rho+\eps$ for every $k\in \N$. Then there exists a sequence of balls $B_{n_k}=B_{ \rho_{n_k} }(x_{n_k})\subseteq\O_{n_k}$. Eventually passing to a subsequence, the sequence $(x_{n_k})$ converges to a point $x_{\infty}$ and the sequence of the translated open set $\O_{n_k}-x_{n_k}$ converges to $\O-x_{\infty}$. Since $B_{r}(0)\subseteq \O_{n_k}-x_{n_k} $ for $r=\rho+\eps$, it turns out that $B_{r}(0)\subseteq \O-x_{\infty}$, i.e. $B_{r}(x_{\infty})\subseteq\O$ which leads to a contradiction.
\end{proof}

\begin{lemm}\label{lem.top2}
Let $\O$ be a connected bounded open set of $\R^n$. There exists a sequence of connected bounded open sets $(\O_n)$ such that $\overline\O_n\subset\O_{n+1}$ and $\bigcup_n\O_n=\O$.
\end{lemm}

\begin{proof}
We construct the sequence by induction. First of all we notice that there exists an integer $\nu_1>0$ such that $\O(\nu_1^{-1})$ contains at least one connected component of $\O$ with Lebesgue measure greater than $\pi\nu_1^{-2}$. Indeed it suffices to choose
$$\nu_1^{-1}\le\min\{ d(y,\pa\O)\ :\ y\in\pa B_r(x)\}\wedge r$$
where $B_r(x)$ is any ball with closure contained in $\O$. Now let $M$ be the number of connected components of $\O(\nu_1^{-1})$ with Lebesgue measure greater than $\pi\nu_1^{-2}$. If $M=1$ we define $ \O_{1}:=\O(\nu_1^{-1})$. Otherwise, since $\O$ is pathwise connected, we can connect the closures of the $M$ connected components with finitely many arcs to define a connected compact set $K\subset\O$. Then, we choose $m$ such that $m>\nu_1$ and $m^{-1}<\inf\{d(x,\pa\O) : x\in K\}$ and we set 
$$\O_{1}:=\{ x\in \O: \ d(x,K)< (2m)^{-1} \}.$$
In both cases $\O_1$ is a connected open set which contains all the connected components of $\O(\nu_1^{-1})$ having Lebesgue measure greater then $\pi\nu_1^{-2}$. Moreover by construction there exists $\nu_2>\nu_1$ such that $\overline \O_1\subseteq \O(\nu^{-1}_2)$.
Replacing $\nu_1$ with $\nu_2$ we can use the previous argument to define $\O_2$ such that $\overline\O_1\subset \O_2$. 
Iterating this argument we eventually define an increasing sequence $\nu_n$ and a sequence of connected open sets $(\O_n)$ such that
$\overline \O_n\subset\O_{n+1}\subset\O$ and $\O_n$ contains all the connected components of $\O(\nu_n^{-1})$ of Lebesgue measure greater than $\pi\nu_{n}^{-2}$.
Since for any $x\in\O$ there exists $r>0$ such that $\overline{B}_r(x)\subset\O$, choosing $\nu_n^{-1}\le \min\{ d(y,\pa\O)\ : y\in\pa B_r(x)\}\wedge r$,
it is easy to show that $x\in\O_n$. Thus $\bigcup_{n} \O_{n}=\O$. 
\end{proof}
In the following lemma we establish  a Bonnesen-type inequality for  sets  $\Omega\in \A_k$ satisfying $\HH^1(\pa \O)<\infty$  (see Theorem 2 in  \cite{Oss79}  when    $\O$ is a simply connected  plane domain bounded by a rectifiable Jordan curve).
 \begin{lemm}\label{lem.coarea2}
Let $\O\in\A_k$ with $\HH^1(\pa \O)<\infty$. 
Then \be\label{eq.coarea}
|\O| \le [2\HH^1(\pa \O)-P(\Omega)+\pi ( k-1) \rho(\O)]\rho(\O).
\ee 
\end{lemm}

\begin{proof}
If $\#\O<\infty$, by Theorem \ref{theo.Nareg} and \eqref{eq.Evans},
$$|\O|\le \left(P(\O)+2\HH^1(\pa\O \cap \O^{(1)})+\pi(k-1)\rho(\O)\right)\rho(\O),$$
and we conclude by \eqref{eq.decomp2}. To prove the general case we denote by $(\O^i)$ the connected components of $\O$ and we set $\O_n:=\bigcup_{i=1}^{n}\O^i$. By the previous step we have
$$|\O_n|\le \big(2\HH^1(\pa \O_n)-P(\Omega_n)+\pi (k-1) \rho(\O_n)\big)\rho(\O_n).$$
Since $\O_n\overset{H^c}{\to}\O$ and $\O_n\to\O$ in the $L^1$-convergence, taking into account Lemma \ref{lem.inradcon} and Lemma \ref{lem.top1}, we can conclude that 
\begin{align*}
|\O|=\lim_{n\to\infty}|\O_n|&\le\big(2\HH^1(\pa\O)-\limsup_n P(\O_n)+\pi(\alpha-1)\rho(\O)\big)\rho(\O)\\
&\le\big(2 \HH^1(\pa\O)-P(\O)+\pi (k-1)\rho(\O)\big)\rho(\O),
\end{align*}
from which the thesis is achieved.
\end{proof}

\begin{theo}\label{theo.convmeas}
Let $(\O_n)\subset\A_k$ be a sequence of equi-bounded open sets with 
$$\sup_n\HH^1(\pa\O_n)<\infty.$$ 
If $\O_n\overset{H^c}{\to}\O$ then $\O\in \A_k$ and $\O_n\to \O$ in the $L^1$-convergence.
If, in addition, either $\sup_n\#\pa\O_n< \infty$ or $\#\O<\infty$ then 
\be
\label{eq.golab}
\HH^1(\pa\O)\le \liminf_n \HH^1(\pa\O_n).
\ee
\end{theo}

\begin{proof} 
We first deal with the case when $\sup_{n}\#\pa\O_n<\infty$, already considered in \cite{ChDo} and \cite{bubu05}. By compactness we can suppose that $\pa\O_n$ converges to some nonempty compact set $K$ which contains $\pa\O$. Then it is easy to show that $\bar\O_n\overset{H}{\to}\O\cup K$, which implies $\chi_{\O_n}\to\chi_\O$ pointwise in $\R^2\setminus K$, where $\chi_E$ denotes the characteristic function of a set $E$. By Theorem \ref{theo.Golab} we have also
\be\label{eq.golabK}
\HH^1(\partial \O)\le\HH^1(K)\le\liminf_{n\to\infty}\HH^1(\partial \O_n)<+\infty,
\ee
which implies \eqref{eq.golab}. In particular, we have $|K|=0$, and $\O_n\to \O$ in the $L^1$ convergence.

We consider now the general case. Let $(\O^i)$ be the connected components of $\O$ and $\eps>0$. There exists an integer $\nu(\eps)$ such that
$$|\O|-\eps<|\bigcup_{i=1}^{\nu(\eps)}\O^i|\le |\O|$$
(when $\#\O<\infty$ we simply choose $\nu(\eps)=\#\O$).
For each $i\le \nu(\eps)$, and for each set $\O^i$, we consider the sequence $(\O^i_n)$ given by Lemma \ref{lem.top2}. By elementary properties of co-Hausdorff convergence there exists $l:=l(n)$ such that
$$\bigcup_{i}^{\nu(\eps)}\overline{\O^i_n}\subset\O_{l}.$$
Let's denote by $\widetilde\O^i_{l}$ the connected component of $\O_{l}$ which contains $\overline{\O^i_n}$ (eventually $\widetilde{\O}^h_{l}=\widetilde{\O}^s_{l}$), and define
$$\widetilde\O_{l}:=\bigcup_{i=1}^{\nu(\eps)} \widetilde\O^i_{l}.$$
By compactness, passing eventually to a subsequence, there exists $\widetilde{\O}\in\A_k$ such that $\widetilde\O_{l}\overset{H^c}{\to}\widetilde\O$.
Moreover, since $\widetilde\O_l\in\A_k$, $\sup_l\#\widetilde\O_l\le\nu(\eps)$, and by Lemma \ref{lem.top1} we have
$$\sup_l \HH^1( \pa \widetilde \O_{l})\le\sup_l \HH^1(\pa\O_{l})<\infty,$$
we can apply the first part of the proof to conclude that $\widetilde\O_l\to \widetilde\O$ in the $L^1$-convergence. If $\#\O<\infty$ an easy argument shows that $\widetilde\O$ must be equal to $\O$ and that \eqref{eq.golabK} holds with $K$ the Hausdorff limit of $(\pa\widetilde\O_l)$. 
In particular \eqref{eq.golab} holds.
Otherwise we consider the set $\O^R_l$ of those connected components of $\O_{l}$ that have been neglected in the definition of $\widetilde\O_l$, that is
$$\O^R_{l}:=\O_{l}\setminus\widetilde\O_{l}.$$ 
Passing to a subsequence we can suppose that $\O^R_l\overset{H^c}{\to}\O^R$, for some open set $\O^R\in\A_k$. Moreover since $|\widetilde\O|>|\O|-\eps$, $\O^R\cap\tilde\O=\emptyset$ and $\O^R\subset\O$ we have also $|\O^R|\le\eps$. This implies $\rho(\O^R)\le\sqrt{\pi^{-1}\eps}$ and by Lemma \ref{lem.inradcon},
$$\lim_{l\to\infty}\rho(\O^R_{l})\le \sqrt{\pi^{-1}\eps}.$$ 
Finally, by Lemma \ref{lem.coarea2}, we have
\begin{align*}
|\O|&\le\liminf_{n\to\infty}|\O_{n}|\le \limsup_{l\to\infty} ( |\widetilde\O_{l}|+|\O^R_{l}|)= |\widetilde \O|+\limsup_{l\to\infty}|\O^R_{l}|\le |\O|+o(\eps).
\end{align*}
Since $\eps$ was arbitrary this shows that
$$
\liminf_{n\to\infty}|\O_n|=|\O|,
$$
and the thesis is easily achieved.
\end{proof} 

As an application of the previous theorem we prove the following fact.

\begin{coro}
Let $\O\in\A_k$ with $\HH^1(\pa\O)<\infty$ and $\#\O<\infty$. Then it holds
$$\HH^1(\pa\O\cap\O^{(0)})\le\HH^1(\pa\O\cap\O^{(1)}).$$
\end{coro}

\begin{proof}
By Theorem \ref{theo.approxim} we can consider a sequence $(A_n)\in\A_k$ of Lipschitz sets such that $A_n\overset{H^c}{\to}\O$ and $P(A_n)\to P(\O)+2\HH^1(\pa\O\cap \O^{(1)})<\infty$. Then, by Theorem \ref{theo.convmeas}, we conclude
$$\HH^1(\pa\O)\le \lim_{n\to\infty} P(A_n)\le P(\O)+2\HH^1(\pa\O\cap \O^{(1)}),$$
which easily implies the thesis, using \eqref{eq.decomp}.
\end{proof}

\begin{rem} We remark the fact that the inequality 
$$\lim_{n\to\infty}P(A_n)\ge\HH^1(\pa \O)$$
is not in general satisfied when $\#\O=\infty$, see also Remark \ref{ex.ce}.
\end{rem}

\section{Existence of relaxed solutions}\label{sexis}

Our next result generalizes the estimate 
$F_{1/2}(\O)\ge3^{-1/2}$, proved in \cite{polya60} for the class $\A_{convex}$, to the class $\A_k$. 

\begin{theo}\label{theo.Polya}
For every $\O\in\A_k$ set of finite perimeter we have
\be\label{eq.polro}
\frac{T^{1/2}(\O)}{|\O|^{3/2}}\ge\frac{3^{-1/2}}{\left(2\HH^1(\pa\O)-P(\O)+2\pi(k-1)\rho(\O)\right)}. 
\ee
\end{theo}

\begin{proof}
Without loss of generality we may assume that $\HH^1(\pa\O)<\infty$ and we set $\rho:=\rho(\O)$. First we consider the case $\#\O<\infty$. We define 
$$G(t):=\int_{0}^{t}\frac{A(t)}{L(t)}dt, \quad u(x):=G(d(x,\pa\O)).$$
Notice that, since for any $t\in (0,\rho)$ it holds $L(t)\ge\HH^1(\partial \O(t))\ge P(\O(t))$,
by isoperimetric inequality \eqref{isoper} we have
$$\frac{A(t)}{L(t)}=\frac{|\Omega(t)|^{1/2}}{L(t)}A^{1/2}(t)\le\frac{|\Omega(t)|^{1/2}}{P(\Omega(t))}A^{1/2}(t)\le\frac{|B_1|^{1/2}}{P(B_1)}A^{1/2}(t).$$
In particular, since $A$ is bounded, we get that $L^{-1}A$ is summable on $(0,\rho)$ and $G$ is a Lipschitz function on in the interval $(0,\rho)$. Thus $u\in H^1_{0}(\O)$.
Using \eqref{vartor} and \eqref{eq.boundL} we have
\begin{align*}
T(\O)&\ge \frac{\left(\int_{\O}udx\right)^{2}}{\int_{\O}|\nabla u|^{2}dx}\ge\frac{\left(\int_{0}^{\rho}G(t)L(t)dt\right)^2}{\int_0^\rho (G'(t))^{2}L(t)dt}\ge\int_0^\rho\frac{(A(t))^{2}}{L(t)}\,dt=\int_{0}^\rho \frac{A^2(t)L(t)}{L^{2}(t)}dt\\
&\ge\frac{1}{(P(\O)+2\HH^1(\partial \O\cap\O^{(1)})+2\pi(k-1)\rho)^2}\int_0^\rho A^2(t)L(t)\,dt.
\end{align*}
Since $A\in W^{1,\infty}(0,\rho(\O))$ by Corollary \ref{theo.Nareg} then, set $\psi(s)=s^2,$ we have that the function $\psi\circ A\in W^{1,\infty}(0,\rho(\O))$, so that
$$\int_0^\rho A^2(t)L(t)\,dt=-\int_0^{\rho}A^2(t)A'(t)\,dt=-\frac13\left[A^3(t)\right]_0^{\rho(\O)}=\frac13|\O|^3.$$
Thus
\be\label{eq.polro1}
\frac{T(\O)}{|\O|^3}\ge\frac{1}{3(P(\O)+2\HH^1(\partial \O\cap\O^{(1)})+2\pi(k-1)\rho)^2}.
\ee
Taking into account \eqref{eq.decomp2} we get
$$
\frac{T(\O)}{|\O|^3}\ge\frac{1}{3(2\HH^1(\pa\O)-P(\O)+2\pi(k-1)\rho)^2}.
$$
To prove the general case, let $\O_n$ be defined as in Lemma \ref{lem.top1}. Since $\#\O_n<\infty$ and $\O_n\in \A_k$, by the first part of this proof we have that
$$\frac{T(\O_n)}{|\O_n|^3}\left(2\HH^1(\pa\O_n)-P(\O_n)+2\pi(k-1)\rho_n\right)^2\ge\frac1{3},$$
where $\rho_n:=\rho(\O_n)$.
When $n\to\infty$ we have $|\O_n|\to |\O|$,  $\rho_n\to\rho$ by Lemma \ref{lem.inradcon} and $T(\O_n)\to T(\O)$ by Theorem \ref{theo.Sve}. Hence, passing to the $\limsup$ in the previous inequality and using Lemma \ref{lem.top1}, we get \eqref{eq.polro}.
\end{proof}

\begin{rem}\label{rem.polya}
Note that, in the special case of $\O\in\A_k$ and $\#\O<\infty$, we have the improved estimate \eqref{eq.polro1}.
Moreover, if $k=0,1$,  \eqref{eq.polro} implies
\be\label{eq.polyagen}
F_{1/2}(\O)\ge\frac{3^{-1/2}P(\O)}{2\HH^1(\pa\O)-P(\O)}\, ,
\ee
while, if $k>1$, we can use the inequality $2\pi\rho(\O)\le P(\O)$ (which can be easily derived from \eqref{isoper}), to obtain
\be\label{eq.polyagenk}
F_{1/2}(\O)\ge\frac{3^{-1/2}P(\O)}{2\HH^1(\pa\O)+(k-2)P(\O)}\;.
\ee
\end{rem}

As a consequence of Theorem \ref{theo.Polya}, and using the well known fact that for a Lipschitz open set $\O$ it holds $P(\O)=\HH^1(\pa\O)$, we have the following main results.

\begin{coro}\label{coro.polya} For every $q\le1/2$ we have
\be\label{eq.m01}
m_{1/2,0}=m_{1/2,1}=3^{-1/2}
\ee
and the value $3^{-1/2}$ is asymptotically reached by a sequence of long thin rectangles. More in general, for $k\ge 1$, it holds
\be\label{eq.boundkq}
m_{q,k}\ge (8\pi)^{1/2-q}(3^{1/2}k)^{-1}
\ee
and the sequence $(m_{q,k})$ decreases to zero as $k\to \infty$.
\end{coro}

\begin{proof}
By inequality \eqref{eq.polyagen} we have that $m_{1/2,0}, m_{1/2,1}\ge 3^{-1/2}$. Moreover the computations made in \cite{BBP20} show that the value $3^{-1/2}$ is asymptotically reached by a sequence of long thin rectangles, that are clearly in $\A_0$. Thus, being $A_0\subset\A_1$, \eqref{eq.m01} holds. To prove \eqref{eq.boundkq} it is enough to notice that
$$F_q(\O)=F_{1/2}(\O)\left(\frac{T(\O)}{|\O|^{2}}\right)^{q-1/2}$$
and apply \eqref{eq.polyagenk} together with the Saint-Venant inequality \eqref{stven}. 
Finally to prove that $m_{q,k}\to 0$ as $k\to \infty$, it is enough to consider the sequence $(\O_{1,n})$ defined in Theorem 2.1 of \cite{BBP20}, taking into account that $\O_{1,n}\in \A_k$ for $k$ big enough.
\end{proof}

We now introduce a relaxed functional $\F_{q,k}$. 
More precisely, for $\O\in\A_k$ we denote by $\mathcal{O}_k(\O)$ the class of equi-bounded sequences of Lipschitz sets in $\A_k$ which converge to $\O$ in the sense of co-Hausdorff and we define $\F_{q,k}$ as follows:
$$\F_{q,k}(\O):=\inf\left\{\liminf_{n\to\infty} F_q(\O_{n}): \ (\O_n)\in\mathcal{O}_k(\O) \right\}.$$
It is straightforward to verify that $\F_{q,k}$ is translation invariant and scaling free. 
As already mentioned in the introduction, when $q<1/2$, we prove the existence of a minimizer for $\F_{q,k}$. We notice this relaxation procedure can be made on the perimeter term only. More precisely, defining
$$\PP_k(\O):=\inf \left\{\liminf_{n\to\infty} P(\O_{n})\ :\ (\O_n)\in\mathcal{O}_k(\O)\right\},$$
the following proposition holds.
\begin{prop}\label{prop.PPkF}
For every $\O\in\A_k$ we have
$$\F_{q,k}(\O)=\frac{\PP_k(\O)T^{q}(\O)}{|\O|^{2q+1/2}}.$$
\end{prop}
\begin{proof}
Fix $\eps>0$. Suppose that $\infty>\PP_k(\O)+\eps\ge\lim_n P(\O_n)$, for some $(\O_n)\in\mathcal{O}_{k}(\O)$. By Theorems \ref{theo.Sve} and \ref{theo.convmeas}, we have
$$
\frac{(\PP_k(\O)+\eps)T^{q}(\O)}{|\O|^{2q+1/2}}\ge\lim_n\left(\frac{P(\O_n)T^q(\O_n)}{|\O_n|^{2q+1/2}}\right)\ge \F_{q,k}(\O),
$$
and since $\eps$ is arbitrary we obtain the $\le$ inequality.
Similarly, to prove the opposite inequality assume $\lim_n F_q(\O_n)\le \F_{q,k}(\O)+\eps<\infty$, for some sequence $(\O_n)\in \mathcal{O}_k(\O)$. Let $D$ be a compact set which contains each $\O_n$. Thanks to Theorem \ref{theo.Sve}, we have that $T(\O_n)\to T(\O)$ and, since $P(\O_n)=\HH^1(\O_n)$, we have also
$$\sup_n\HH^1(\pa\O_n)=\sup_n\left( \frac{F_q(\O_n)|\O_n|^{2q+1/2}}{\displaystyle{T^q}(\O_n)}\right)\le
\sup_n \left(\frac{F_q(\O_n)|D|^{2q+1/2}}{\displaystyle{T^q}(\O_n)}\right)<+\infty.$$
Applying again Theorem \ref{theo.convmeas} we have $|\O_n|\to|\O|$ and we can conclude
$$\frac{\PP_k(\O)T^q(\O)}{|\O|^{2q+1/2}}\le\lim_n F_q(\O_n)\le\F_{q,k}(\O)+\eps,$$
which implies the $\ge$ inequality as $\eps\to 0$.
\end{proof}

The perimeter $\PP_k$ satisfies the following properties.

\begin{prop}\label{prop.Pk}
For every $\O\in\A_k$ of finite perimeter we have
\be\label{eq.PPk}
P(\O)\le\PP_k(\O)\le 2\HH^1(\pa\O)-P(\O).
\ee
Moreover if $\#\O<\infty$ and $\HH^1(\pa\O)<+\infty$ it holds
\be\label{eq.PPkH1}
\HH^1(\partial\O)\le\PP_k(\O)\le P(\O)+2\HH^1(\partial \O \cap \O^{(1)})
\ee
and $P(\O)=\PP_k(\O)$ if and only if $P(\O)=\HH^1(\pa\O)$.
\end{prop}

\begin{proof}
Taking into account Theorem \ref{theo.convmeas} and lower semicontinuity of the perimeter with respect to the $L^1$-convergence we have $\PP_k(\O)\ge P(\O)$.
To prove the right-hand inequalities in \eqref{eq.PPk} and \eqref{eq.PPkH1} it is sufficient to take the sequence $(A_n)$ given by Theorem \ref{theo.approxim}. Finally, when $\#\O<\infty$, the inequality $\HH^1(\pa\O)\le \PP_k(\O)$ follows by Theorem \ref{theo.convmeas}. 
\end{proof}

\begin{rem} \label{ex.ce}
If we remove the assumption $\#\O<\infty$, then \eqref{eq.PPkH1} is no longer true. For instance, we can slightly modify the Example $3.53$ in \cite{AFP} to define $\O\in\ \A_0$ such that $P(\O),\PP_0(\O)<\infty$ while $\HH^1(\pa\O)=\infty$. More precisely let $(q_n)$ be an enumeration of $\mathbb{Q}^2\cap B_1(0)$ and $(r_n)\subset(0,\eps)$ be a decreasing sequence such that
$\sum_n 2\pi r_n\le 1$. We recursively define the following sequence of open sets.
Let 
$$\O_0:=B_{r_0}(q_0),\ \O_{n+1}:=\O_n\cup B_{s_n}(q_{h_n}),$$
where 
$$h_n:=\inf\{k: q_k \in\overline\O_n^c\},\quad s_n:=r_{n+1}\wedge\sup\{r_k: B_{r_k}(q_{h_n})\cap \O_n=\emptyset\}.$$ 
Finally let $\O=\bigcup_n\O_n$. By construction $\O_n\overset{H^c}{\to}\O$ and since $\O_n\in\A_0$ for all $n$, we have also $\O\in\A_0$. Moreover we notice that $P(\O)\le 1$ and it is easy to verify that the two dimensional Lebesgue measure of $\pa\O$ is positive, which implies $\HH^1(\pa\O)=\infty$. Finally, since the sequence $(\O_n)\in \mathcal{O}_0(\O)$, we have also $\PP_0(\O)\le 1$. 
\end{rem}

Next we prove that the relaxed functional $\F_{q,k}$ agrees with $F_q$ on the class of Lipschitz open sets in $\A_k$.
\begin{coro}\label{coro.Frel}
For every $\O\in\A_k$ we have
\be\label{eq.FgF}
\F_{q,k}(\O)\ge F_q(\O).
\ee
If, in addition, $P(\O)=\HH^1(\pa\O)$ then we have
\be \label{eq.FgF1}
F_q(\O)=\F_{q,k}(\O).
\ee
In particular $\F_{q,k}$ and $F_q$ coincide on the class of Lipschitz sets and it holds
\be\label{eq.infrel}
m_{q,k}=\inf\{\F_{q,k}(\O)\ :\ \O\in\A_k\}.
\ee
\end{coro}

\begin{proof}
The inequalities \eqref{eq.FgF} and \eqref{eq.FgF1} follow by Proposition \ref{prop.PPkF} and \eqref{eq.PPk}. The last part of the theorem follows as a general property of relaxed functionals. 
\end{proof}

\begin{lemm}\label{lem.coninf}
For every Lipschitz set $\O\in\A_k$, there exists a sequence of connected open sets $(\O_n)\subset\A_k$ such that
$$P(\O_n)=\HH^1(\pa\O_n)\qquad\text{and}\qquad\lim_{n\to\infty}F_q(\O_n)=F_q(\O).$$
\end{lemm}

\begin{proof}
Since $\O$ is a bounded Lipschitz set we necessarily have $\#\O<\infty$. If $\O$ is connected we can take $\O_n$ to be constantly equal to $\O$. Suppose instead that $\#\O=2$ and let $\O^1$ and $\O^2$ be the connected components of $\O$. Since $\O$ is Lipschitz there exist $x_1\in\pa \O^1, x_2\in\pa\O^2$ such that
$$
0<d:=d(x_1,x_2)=\inf\{d(w,v): \ v\in\O^1,\ w\in\O^2\}.
$$
Define 
$$\O^{2}_n:=\O^2-\left(1-\frac 1 n\right)(x_2-x_1).$$
Clearly we have $\overline{\O^2_{n}}\cap \overline{\O^1}=\emptyset$ for every $n\ge 1$ and $\O^{2}_1=\O^2$. 
We set $$x_n=x_2-\left(1-\frac 1 n\right)(x_2-x_1).$$
Now we can join $x_1$ and $x_n$ through a segment $\Sigma_n$. By using the fact that the boundary of both $\pa\O^1$ and $\pa\O^2_n$ are represented as the graph of a Lipschitz functions in a neighborhood of $x_1$ and $x_n$ respectively, then the thin open channel
$$C_\eps:=\{ x\in\R^2\setminus\overline\O^1\cup\overline\O^2_n\ :\ d(x,\Sigma_n)<\eps\}$$
of thickness $\eps:=\eps(n)$ is such that the set
$$\O_n:=\O^1\cup \O^2_n\cup C_{\eps}$$
belongs to $\A_k$, it is connected and $P(\O_n)=\HH^1(\pa\O_n)$. The following identities are then verified
$$|\O_n|\to|\O|,\quad T(\O_n)\to T(\O),\quad P(\O_n)\approx P(\O^1)+P(\O^2)+\frac{2\eps}{n},$$ 
so that $F_q(\O_n)\to F_q(\O)$ (notice that this does not imply $\O_n\to\O$). The general case is achieved by induction on $\#\O$. More precisely suppose $\#\O=N+1$. Let $(\O^i)$ be the connected components of $\O$. By induction we have
$$F_q(\O^1\cup\dots \cup\O^N)=\lim_{n\to\infty}F_q(\O'_n),$$
for a sequence $(\O'_n)\subset\A_k$ of connected open sets satisfying $P(\O'_n)=\HH^1(\pa\O'_n)$. 
Using the fact that, being $\O$ Lipschitz, the value of $F_q(\O)$ do not change if we translate (possibly in different direction and with different magnitude) each connected component of $\O$, being careful to avoid intersections, we can suppose $\overline{\O}^{N+1}$ to have a positive distance from $\overline{\O}'_n$, as $n$ is large enough.
We then apply the previous step to define a sequence of connected open sets $\O_{n,m}\in\ A_k$ such that $P(\O_{n,m})=\HH^1(\pa\O_{n,m})$ and 
$$F_q(\O_{n,m})\to F_q(\O'_n\cup \O^{N+1}),$$
as $m\to\infty$. Using a diagonal argument we achieve the thesis.
\end{proof}

We finally show the existence of a relaxed solution to the minimization problem of $\F_{q,k}$ in $\A_k$ when $q<1/2$.
\begin{theo}\label{theo.exis}
For $q<1/2$ there exists a nonempty bounded open set $\O^{\star}\in\A_k$ minimizing the functional $\F_{q,k}$ such that $\HH^1(\pa\O^\star)<\infty$. 
\end{theo}

\begin{proof}
Let $(\widetilde\O_n)\subset\A_k$ be a sequence of Lipschitz sets such that
$$
\lim_{n\to\infty} F_q(\widetilde\O_n)=m_{q,k}.
$$
Applying Lemma \ref{lem.coninf} and \eqref{eq.FgF1}, we can easily replace the sequence $(\widetilde\O_n)$ with a sequence $(\O_n)\subset\A_k$ of connected (not necessarily Lipschitz) open sets, satisfying $\HH^1(\O_n)=P(\O_n)$ and such that
$$\lim_{n\to\infty} F_q(\O_n)=\lim_{n\to\infty} F_q(\widetilde\O_n)=m_{q,k}.$$
Eventually using the translation invariance of $F_q$ and possibly rescaling the sequence $(\O_n)$, we can assume that $(\O_n)$ is equi-bounded and
\be\label{ipinf}
\HH^1(\O_n)=P(\O_n)=1.
\ee 
By compactness, up to subsequences, there exists an open sets $\O^\star\in\A_k$ such that $\O_n\overset{H^c}{\to}\O^\star$.
By \eqref{eq.infrel} we have
$$m_{q,k}\le \F_{q,k}(\O^\star).$$
Let us prove the opposite inequality. We notice that, by Theorem \ref{theo.approxim} and \eqref{ipinf}, for every $n$ there exists a sequence $(A_{n,m})_m\subset\A_k$ of Lipschitz sets, such that, as $m\to\infty$,
$$
P(A_{n,m})\to P(\O_n)\ \text{and}\ |A_{n,m}|\to |\O_n|.
$$
By Theorem \ref{theo.Sve}, we have also $T(A_{n,m})\to T(\O_n)$ as $m\to \infty$.
Thus
$$F_q(\O_n)=\lim_{m\to\infty}F_q(A_{n,m}).$$ 
A standard diagonal argument allows us to define a subsequence $A_{n,m_n}\in\mathcal{O}_k(\O^\star)$. Then we have
$$\F_{q,k}(\O^\star)\le \lim_n F_{q}(A_{n,m_n})=\lim_{n}F_{q}(\O_n)= m_{q,k}.
$$
Hence $\O^\star$ is a minimum for $\F_{q,k}$.
Moreover, notice that there exists a compact set $K$ containing $\pa\O^\star$ such that, up to a subsequence, $\pa\O_n\overset{H}{\to} K$.
So, being $\O_n$ connected, we have
$$
\sup_n\#\pa\O_n<\infty, 
$$ 
and by Theorem \ref{theo.Golab},
$$
\HH^1(\pa\O^\star)\le \HH^1(K)\le \liminf_{n\to\infty}\HH^1(\O_n)\le 1.
$$
To conclude we have only to show that $\O^*$ is nonempty. Notice that for $n$ big enough there exists $C>0$ such that $F_q(\O_n)< C$. Thus we have
\be \label{eq.final1}
C>F_q(\O_n)=\frac{T^q(\O_n)}{|\O_n|^{2q+1/2}}=\left(\frac{T(\O_n)}{|\O_n|^{3}}\right)^{q}|\O_n|^{q-1/2}\ge \frac{1}{|\O_n|^{1/2-q}(\sqrt{3}k)^{2q}}\;,
\ee
where the last inequality follows by  \eqref{eq.polyagen}, using \eqref{ipinf}. By \eqref{eq.coarea} we have also
\be \label{eq.final2}
|\O_n|\le (1+\pi(k-1)\rho(\O_n))\rho(\O_n).
\ee
Combining \eqref{eq.final1}, \eqref{eq.final2} and the assumption $q<1/2$, we conclude that the sequence of inradius $(\rho(\O_n))$ must be bounded from below by some positive constant. By Lemma \ref{lem.inradcon}, $\O^{\star}$ is nonempty.
\end{proof}

\section{Conclusions}\label{sconc}

We have seen that in the planar case the topological constraint present in classes $\A_k$ is strong enough to ensure the existence of at least a relaxed optimizer. In higher dimensions this is no longer true and easy examples show that it is possible to construct sequences $(\O_n)$ in $\A_k$ with $P(\O_n)$ bounded and $T(\O_n)\to0$. This suggests that in higher dimensions stronger constraints need to be imposed in order to have well posed optimization problems.

Another interesting issue is the analysis of the same kind of questions when the exponent $2$ is replaced by a general $p>1$ in \eqref{vartor}; the torsional rigidity $T(\O)$ then becomes the $p$-torsion $T_p(\O)$ and it would be interesting to see how our results depend on the exponent $p$ and if in this case the analysis in dimensions higher than two is possible.

Finally, shape functionals $F(\O)$ involving quantities other than perimeter and torsional rigidity are interesting to be studied: we point out some recent results in \cite{bbp20},\cite{FtLa} and references therein. However, to our knowledge, the study of these shape functionals under topological constraints as the ones of classes $\A_k$ is still missing.

\bigskip

\noindent{\bf Acknowledgments.} The work of GB is part of the project 2017TEXA3H {\it``Gradient flows, Optimal Transport and Metric Measure Structures''} funded by the Italian Ministry of Research and University. The authors are member of the Gruppo Nazionale per l'Analisi Matematica, la Probabilit\`a e le loro Applicazioni (GNAMPA) of the Istituto Nazionale di Alta Matematica (INdAM).

\bigskip
{\small\noindent
Luca Briani:
Dipartimento di Matematica,
Universit\`a di Pisa\\
Largo B. Pontecorvo 5,
56127 Pisa - ITALY\\
{\tt luca.briani@phd.unipi.it}

\bigskip\noindent
Giuseppe Buttazzo:
Dipartimento di Matematica,
Universit\`a di Pisa\\
Largo B. Pontecorvo 5,
56127 Pisa - ITALY\\
{\tt giuseppe.buttazzo@dm.unipi.it}\\
{\tt http://www.dm.unipi.it/pages/buttazzo/}

\bigskip\noindent
Francesca Prinari:
Dipartimento di Matematica e Informatica,
Universit\`a di Ferrara\\
Via Machiavelli 30,
44121 Ferrara - ITALY\\
{\tt francescaagnese.prinari@unife.it}\\
{\tt http://docente.unife.it/francescaagnese.prinari/}}

\end{document}